\documentclass{article}

\usepackage{amsmath}
\usepackage{amsthm}
\usepackage{amssymb}
\usepackage{amsfonts}
\usepackage{enumerate}
\usepackage{graphicx}
\usepackage{xcolor}
\usepackage{xypic}
\usepackage{upgreek}

\newcommand{\bmu}{\mathbf{\upmu}}

\newcommand{\Z}{\mathbb{Z}}

\renewcommand{\P}{\mathbb{P}}
\newcommand{\F}{\mathbb{F}}
\newcommand{\G}{\mathbb{G}}

\newcommand{\Q}{\mathbb{Q}}
\newcommand{\C}{\mathbb{C}}

\newcommand{\psum}[1]{ \underset{#1}{{\sum}^\prime}}

\newtheorem{Theo}{Theorem}
\newtheorem{Prop}[Theo]{Proposition}
\newtheorem{Defn}[Theo]{Definition}
\newtheorem{Lemma}[Theo]{Lemma}
\newtheorem{Cor}[Theo]{Corollary}

\newtheorem*{MainThm}{Main Theorem}

\theoremstyle{remark}
\newtheorem*{Remark}{Remark}
\newtheorem{Exam}{Example}

\begin{document}

\title{Class number formulas via 2-isogenies \\ of elliptic curves}

\author{Cam McLeman\\ University of Michigan-Flint \and Christopher
  Rasmussen\\ Wesleyan University}

\date{}

\maketitle

\let\thefootnote\relax\footnotetext{This article has been accepted for
  publication in the Bulletin of the London Mathematical
  Society. However, it does not reflect any changes or corrections
  between acceptance and publication; the two versions may be
  different.}

\let\thefootnote\relax

\let\thefootnote\relax\footnotetext{This research was supported in
  part by the Van Vleck Fund at Wesleyan University.}

\begin{abstract}
  A classical result of Dirichlet shows that certain elementary
  character sums compute class numbers of quadratic imaginary number
  fields. We obtain analogous relations between class numbers and a
  weighted character sum associated to a $2$-isogeny of elliptic
  curves.
\end{abstract}

\section{Introduction} 
\label{intro}

We begin by recalling a famous result of Dirichlet which calculates
the class number of a quadratic imaginary number field of prime discriminant
via a finite sum.  Throughout, we let~$p > 3$ denote a prime
number. Let $(\frac{\cdot}{p})$ denote the usual Legendre symbol on
$\F_p^\times$:
\begin{equation*}
\left( \frac{a}{p} \right) = \left\{ \begin{array}{rcl}
+1 & & a \in \F_p^{\times 2} \\
-1 & & a \not\in \F_p^{\times 2}.
\end{array}
\right.
\end{equation*}
The symbol is extended to all of $\Z$ via the reduction map $\Z \to
\F_p$, and the definition $(\frac{a}{p}) = 0$ when $(a,p) > 1$. We let
$h_p$ denote the class number of $\Q(\sqrt{-p})$. For notational
convenience, we set
\begin{equation*}
h_p^* := \left\{ \begin{array}{rcl} 0\phantom{{}_p} & & p \equiv 1
      \pmod{4} \\
h_p & & p \equiv 3 \pmod{4}.
\end{array}
\right.
\end{equation*}
The following consequence of Dirichlet's class number formula (e.g.,
\cite[Ch.~6]{Davenport:2000}) is well-known:
\begin{Theo}[Dirichlet]\label{ACNF}
For any prime $p > 3$, 
\begin{equation}\label{eq:Dirichlet}
-\frac{1}{p} \sum_{x=1}^{p-1} x \bigl( \tfrac{x}{p} \bigr) = h_p^*.
\end{equation}
\end{Theo}

We consider the following point of view for Dirichlet's result. The
$\F_p$-rational morphism~$\phi \colon \G_m \to \G_m$ defined by
$\phi(x) = x^2$ partitions the points of $\G_m(\F_p)\cong \F_p^\times$
into two sets -- those that are the image of an $\F_p$-rational point
of the domain (i.e., the quadratic residues), and those that are not
(the non-residues). The character $\chi_\phi := ( \frac{ \cdot }{p} )$
is now precisely the natural identification of the cokernel of $\phi$
with $\{\pm 1\}$ which makes the following sequence exact:
\begin{equation*}
  \xymatrix{
\G_m(\F_p) \ar[r]^\phi & \G_m(\F_p) \ar^(.55){\chi_\phi}[r] & \{ \pm1 \}
\ar[r] & 0\vphantom{0_n}  
}.
\end{equation*}

Note that it is the properties of $\phi$, not the underlying algebraic
group $\G_m$, which allow this construction. This paper demonstrates
that an analogous procedure, arising from a different morphism of
algebraic groups, yields new character sums with similar arithmetic
properties. Let $\tau$ be a degree $2$ isogeny of elliptic curves
defined over $\F_p$. We define a weighted character sum $S_\tau$,
analogous to the sum appearing in \eqref{eq:Dirichlet}. The quantity
$S_\tau$ is shown to be divisible by~$p$, and a strong relationship
between $S_\tau$ and $h_p^*$ is established.

\begin{Remark}
  Throughout this article, we study sums of the form $\sum g(x)
  \chi(x)$, where \emph{a priori} the values of $g(x)$ lie in the
  finite field $\F_p$. We use the following convention, so as to view
  the value of the sum as an integer: Each summand is the scaling of
  the character value $\chi(x) \in \bmu_2(\C) = \{ \pm 1 \}$ by the
  unique integral lift of $g(x)$ in the range $[0,p)$. For clarity, we
  use braces $\{ \cdot \}$ to denote the lifting $\F_p \to \Z \cap
  [0,p)$ explicitly.

In some proofs, it will be convenient to view sums of the form
$\sum_{x=0}^{p-1}$ interchangeably as sums over $\F_p$ or as sums over
the range $[0,p)$ of integers.  Consequently, there will occasionally
be a mild abuse of notation -- for example, given $a \in \Z$, we may
write $\{ a \}$ to mean $\{\overline{a}\}$, where $\overline{a}$ is the
reduction of $a$ mod $p$.
\end{Remark}

\subsection*{CM Example}
We begin with two typical results, both special cases of the main
theorem. Consider first the elliptic curve
\begin{equation}\label{eq:CMneg2curve}
E/\Q \colon \qquad y^2 = (x+2)(x^2-2),
\end{equation}
which possesses complex multiplication by $\Z[\sqrt{-2}]$. As this
ring possesses elements of absolute norm $2$, there exist
endomorphisms of degree $2$ on $E$. These endomorphisms admit
reductions defined over $\F_p$ whenever $p$ is a prime of good and
ordinary reduction (equivalently, $(\frac{-2}{p}) = 1$). For a
specific choice of $\tau$ (see \S4), we set 
\begin{equation}\label{eq:chitau}
\chi_\tau(P) = \left\{ \begin{array}{rcl}
+1 & & P \in \tau \bigl( E(\F_p) \bigr) \\
-1 & & P \not\in \tau \bigl( E(\F_p) \bigr).
\end{array}
\right.
\end{equation}
This defines a character on $E(\F_p)$, and the following sequence is exact:
\begin{equation*}
  \xymatrix{
    E(\F_p) \ar[r]^\tau & E(\F_p) \ar[r]^(.55){\chi_\tau} & \bmu_2 \ar[r] &
    0\vphantom{E(\F_p)}  
  }.
\end{equation*}
The following is a consequence of the main theorem:
\begin{Prop}\label{-2ex}
  Let $p>3$ be a prime of good and ordinary reduction for the elliptic
  curve~$E$ given in \eqref{eq:CMneg2curve}. Then
\begin{equation}\label{eq:CMneg2}
-\frac{1}{p} \sum_{\substack{P \in E(\F_p) \\ P \neq \infty}} \{ x(P) \} \chi_\tau(P) = h_p^*.
\end{equation}
\end{Prop}

\begin{Remark}
It is not \emph{a priori} clear that the sum in
  \eqref{eq:CMneg2} (nor the sum in \eqref{eq:Dirichlet}, for that
  matter) is divisible by $p$. We note that the two sums
  appearing in \eqref{eq:Dirichlet} and \eqref{eq:CMneg2} are
  \emph{not} the same expressions, even though they both compute
  $h_p^*$. This is immediate from the observation that $E$ and $\G_m$
  need not have the same number of points over $\F_p$. Here is an
explicit example: When $p = 11$, one has
\begin{equation*}
E(\F_{11}) = \{\infty, (7, \pm  4), (8, \pm 2), (9,0)\}, \qquad 
\tau \bigl( E(\F_{11}) \bigr) = \{\infty, (7, \pm 4) \}.
\end{equation*}
Hence, \eqref{eq:CMneg2} evaluates as 
\begin{equation*}
-\frac{1}{11} (7 + 7 - 8 - 8 - 9) = 1,
\end{equation*}
whereas the classical expression \eqref{eq:Dirichlet} yields
\begin{equation*}
-\frac{1}{11} (1 - 2 + 3 + 4 + 5 - 6 - 7 - 8 + 9 - 10) = 1.
\end{equation*}
\end{Remark}

\subsection*{Non-CM Example}
This connection between weighted character sums and class numbers is
not unique to isogenies arising from complex multiplication. Consider
the elliptic curves
\begin{equation*}
\begin{split}
E_1/\Q \colon & \qquad y^2 = x^3 + 2x^2 - \phantom{8}x, \\
E_2/\Q \colon & \qquad y^2 = x^3 - 4x^2 + 8x,
\end{split}
\end{equation*}
and the following isogeny of degree two:
\begin{equation*}
  \tau \colon E_1 \to E_2, \qquad \qquad \tau(x,y) = \left(
    \frac{y^2}{x^2}, -\frac{y(1+x^2)}{x^2} \right).
\end{equation*}
The curves $E_1$ and $E_2$ have good reduction away from $2$. For any
odd prime $p$, $\tau$ induces an isogeny between the reductions of
$E_1$ and $E_2$ over $\F_p$, and this morphism is in fact
$\F_p$-rational. Hence, there exists a homomorphism $\tau \colon
E_1(\F_p) \to E_2(\F_p)$. As in the previous example, we consider the
character
\begin{equation*}
\chi_\tau \colon E_2(\F_p) \to \{ \pm 1 \},
\end{equation*}
where $\chi_\tau(P) = +1$ if and only if $P \in \tau \bigl( E_1(\F_p)
\bigr)$. Again we find a strong relationship between~$h_p^*$ and the
weighted character sum
\begin{equation*}
S_\tau := \sum_{\substack{P \in E_2(\F_p) \\ P \neq \infty }} \bigl\{ x(P) - 2 \bigr\} \chi_\tau(P).
\end{equation*}
The following is a special case of the main theorem (see \S4, Example \ref{ex:intro2}):
\begin{Prop}\label{secondex}
With $E_1$, $E_2$, and $\tau$ as above, and any prime $p>3$, we have $-\frac{1}{p} S_\tau = h_p^*$.
\end{Prop}

\subsection*{Main result}

The main theorem generalizes the previous examples, which each relate
$h_p^*$ to a weighted character sum. Suppose $E_1/\Q$, $E_2/\Q$ are
elliptic curves and $\tau \colon E_1 \to E_2$ is a $\Q$-rational
$2$-isogeny. For any prime $p$ of good reduction, there is an
$\F_p$-rational isogeny $\tau_p \colon E_1/\F_p \to E_2/\F_p$. Hence,
$\tau$ induces a family of isogenies $\{ \tau_p \}$, indexed by the
primes of good reduction. To each of these isogenies, there is an
associated character $\chi_\tau = \chi_{\tau,p}$, and an associated
weighted character sum $S_{\tau,p}$ (defined below). As in the above
examples, the quotient $-\frac{1}{p} S_{\tau,p}$ always approximates
$h_p^*$ well, in the sense that there is an absolute bound for the
error as $p$ varies among all primes of good reduction.

Concretely, let $a, b \in \Z$, and let $E_1/\Q$ and $E_2/\Q$ be the
elliptic curves given by the following Weierstrass models:
\begin{equation}\label{eq:E1E2}
\begin{split}
E_1 \colon & \quad y^2 = x^3 + \phantom{2}ax^2 + bx, \\
E_2 \colon & \quad y^2 = x^3 - 2ax^2 + (a^2-4b)x
\end{split}
\end{equation}
Let $\tau \colon E_1 \rightarrow E_2$ be the explicit isogeny given in
\eqref{eq:tau_formulas}.

\begin{MainThm}
  \emph{ Let $E_1$, $E_2$, and $\tau$ be as above, and let $p>3$ be a
    prime of good reduction for $E_1$ and $E_2$.
\begin{enumerate}[(a)]
\item For all such $p$, the weighted character sum
\begin{equation*}
  S_{\tau, p} := \sum_{\substack{P \in E_2(\F_p) \\ P \neq \infty }}\bigl\{ x(P)-a \bigr\} \chi_\tau(P)
\end{equation*}
is divisible by $p$.
\item $S_{\tau, p}$ approximates $-p h_p^*$ in the following sense:
  the quantity
\begin{equation*}
R_{a,b}(p) = -\frac{1}{p} S_{\tau, p} - h_p^*
\end{equation*}
is bounded in absolute value by a constant $C_\tau$, independent of
$p$.
\item If there exists one $p>|a|$ such that $R_{a,b}(p) = 0$, then
  there exists a set of primes of positive density for which $S_{\tau,
    p} = h_p^*$, determined by explicit congruence conditions.
\end{enumerate}}
\end{MainThm}

The remainder of the paper is organized as follows. In \S2, we collect
relevant facts about $2$-isogenies over finite fields, and compute the
cokernel character $\chi_\tau$ in terms of the Tate pairing. In \S3,
we prove the main theorem. The proof combines classical techniques for
evaluating character sums with formulas deduced from the Tate
pairing. In \S4, we extend the theorem to isogenies not specifically
of the form \eqref{eq:tau_formulas}, in particular the degree $2$
endomorphisms existing on elliptic curves with complex multiplication
by $\Z[\sqrt{-1}]$, $\Z[\sqrt{-2}]$, and $\Z[\sqrt{-7}]$.  An
analogous result for the dual isogeny $\hat{\tau}$ of
\eqref{eq:tau_formulas} is given in the Appendix.

\section{Preliminaries}

Our strategy for the proof of the Main Theorem will be to convert the
characters in the introduction into explicitly computable Legendre
symbols.  For this, we recall some facts about isogenies of degree 2
and the mechanics of the Tate pairing attached to an isogeny.

\subsection*{Degree 2 isogenies}
Let $K$ be an arbitrary field with $\mathrm{char}\, K \neq 2$, and let
$E_1$ and $E_2$ be elliptic curves defined over~$K$. Let $\phi \colon
E_1 \to E_2$ be an isogeny of degree $2$ defined over
$K$. Necessarily, this implies that both $E_1$ and $E_2$ possess
$K$-rational points of order $2$ (generating the kernels of $\phi$ and
the dual~$\hat{\phi}$, respectively). Hence, by appropriate changes of
coordinates, $E_1$ and $E_2$ are isomorphic over~$K$ to elliptic
curves $E_1'$ and $E_2'$, respectively, which each possess Weierstrass
equations of the form~$y^2=f_i(x)$, with $f_1(0) = f_2(0) = 0$. In
fact, we may always simultaneously choose isomorphisms $\alpha_i
\colon E_i \to E_i'$ such that:
\begin{enumerate}[(i)]
\item $E_1'$ has the Weierstrass model $y^2 = x^3 + \phantom{2}ax^2 + bx$ with
$a, b \in K$;
\item $E_2'$ has the Weierstrass model $y^2 = x^3 - 2ax^2 + rx$, with
  $r = a^2 - 4b$;
\item There exists a $K$-rational $2$-isogeny $\tau \colon E_1' \to
  E_2'$ such that the following diagrams commute:
\begin{equation}\label{eq:comm_diag}
\xymatrix{
E_1 \ar[d]_\phi \ar[r]^{\alpha_1} & E_1' \ar[d]^\tau & & E_2
\ar[d]_{\hat{\phi}} \ar[r]^{\alpha_2} & E_2' \ar[d]^{\hat{\tau}} \\
E_2 \ar[r]_{\alpha_2} & E_2' & & E_1 \ar[r]_{\alpha_1} & E_1'
}
\end{equation}
\end{enumerate}
The explicit formulas for $\tau$ and its dual are well-known
\cite[III.4.5]{Silverman:AEC}:
\begin{equation}\label{eq:tau_formulas}
  \tau(x,y) = \left( \frac{y^2}{x^2}, \frac{y (b-x^2) }{x^2} \right),
  \qquad 
\hat{\tau}(x, y) = \left( \frac{y^2}{4x^2}, \frac{y(r-x^2)}{8x^2} \right).
\end{equation}
Placing the isogeny in this form will simplify the computations
involving the Tate pairing.  We will show how to treat more general
2-isogenies in Section 4.

\subsection*{The Tate pairing}
Let $\ell$ and $p$ be prime numbers such that $p \equiv 1
\pmod{\ell}$, and let $E_1$ and $E_2$ be elliptic curves defined
over $\F_p$. Let $\tau \colon E_1 \to E_2$ be an isogeny of degree
$\ell$ defined over $\F_p$, and let $\hat{\tau}$ denote the dual
isogeny. Let us assume that $E_1[\tau] \subseteq E_1(\F_p)$ and
$E_2[\hat{\tau}] \subseteq E_2(\F_p)$, and let $T_1$ and $T_2$
generate the groups $E_1[\tau]$ and $E_2[\hat{\tau}]$, respectively.
The Tate pairing associated to $\tau$ is a function
\begin{equation*}
\psi_\tau \colon \frac{ E_2(\F_p) }{ \tau(E_1(\F_p)) } \times
E_2[\hat{\tau}] \longrightarrow \F_p^\times/\F_p^{\times \ell}
\end{equation*}
which is bilinear and non-degenerate on the left. (See, for example,
\cite[X.1.1]{Silverman:AEC}, which develops the properties of
$\psi_{[m]}$. The proofs for $\psi_\tau$ are essentially identical.)
We can explicitly compute~$\psi_\tau$ as follows.  The points $T_1$
and $T_2$ are of exact order $\ell$, as is any $R \in
\tau^{-1}(T_2)$. Thus, there exist functions $f$ on $E_2$ and $g$ on
$E_1$ whose divisors are:
\begin{equation}\label{eq:divisors}
\begin{split}
(f) & = \ell(T_2) - \ell(\infty), \\
(g) & = \tau^* \bigl( (T_2) - (\infty) \bigr) = \sum_{i=0}^{\ell - 1}
(R + iT_1) - \sum_{i=0}^{\ell-1} (iT_1).
\end{split}
\end{equation}
Scaling either $f$ or $g$ by a constant if necessary, we have an
equality of functions $f \circ \tau = g^\ell$. For a point $S\in
E_2(\F_p)$, let $[S]$ denote its image in $E_2(\F_p)/\tau(E_1(\F_p))$.
If $S \in E_2(\F_p)$ and $[S]\neq[T_2],[\infty]$, we define
\begin{equation*}
\psi_\tau \bigl( [S], T_2 \bigr) := f(S) \pmod{ \F_p^{\times \ell}}. 
\end{equation*}
In case $[S] \in \{[T_2], [\infty]\}$, we choose a point $Q \not\in
\{T_2, \infty \}$ on $E_2$ and set
\begin{equation*}
\psi_\tau \bigl( [S], T_2 \bigr) := \frac{ f(S + Q) }{ f(Q) }
\pmod{ \F_p^{\times \ell}}.
\end{equation*}
As $E_2[\hat{\tau}]$ is generated by $T_2$, we may recover the entire
pairing by bilinearity, since we have~$\psi_\tau([S], kT_2) = \psi_\tau([S], 
T_2)^k$. Thus, the following definition is complete:
\begin{equation*}
\psi_\tau( [S], kT_2) := \left\{ \begin{array}{ccl}
f(S)^k & & [S] \notin \{[T_2], [\infty] \} \\
\left( \frac{ f(S + Q) }{f(Q)} \right)^k & & [S] \in \{[T_2], [\infty]
\}.
\end{array} \right.
\end{equation*}

\subsection*{Formulas for $2$-isogenies}
We now specialize to the case of the degree $2$ isogenies $\tau$ and
$\hat{\tau}$ given in \eqref{eq:tau_formulas}.
\begin{Defn}
Define a character $\chi_{\tau} \colon E_2(\F_p) \to \bmu_2$ by
\begin{equation*}
\begin{split}
\chi_\tau(P) = \left\{ \begin{array}{rcl}
+1 & & P \in \tau \bigl( E_1(\F_p) \bigr) \\
-1 & & P \notin \tau \bigl( E_1(\F_p) \bigr).
\end{array} \right.
\end{split}
\end{equation*}
\end{Defn}
We compute this character explicitly in terms of the pairing
$\psi_\tau$. This allows us to replace~$\chi_\tau$ with expressions
involving the Legendre symbol, and so evaluate the weighted character
sum~$S_\tau$ from the introduction. We let $T_1$ and $T_2$
respectively denote the point $(0,0)\in E_1(\F_p)$, which generates
$E_1[\tau]$, and the point $(0,0)\in E_2(\F_p)$, which generates
$E_2[\hat{\tau}]$.
\begin{Prop}\label{prop:chi_psi}
Suppose $\tau \colon E_1 \to E_2$ is of the form
\eqref{eq:tau_formulas}, and let $P = (x,y) \in E_2(\F_p)$.  Then
\begin{equation*}
\chi_\tau(P) = \psi_\tau([P], T_2) = \left\{
\begin{array}{ccl}
\bigl( \tfrac{x}{p} \bigr) & & [P] \neq [T_2], [\infty] \\
\bigl( \tfrac{r}{p} \bigr) & & [P] = [T_2] \\
1 & & [P] = [\infty],
\end{array}
\right.
\end{equation*}
where in the second equality we canonically identify
$\F_p^\times/\F_p^{\times 2}$ with $\bmu_2$ via the Legendre symbol.
\end{Prop}
\begin{proof}
  For the statement that $\chi(\cdot)=\psi_{\tau}(\cdot,T_2)$, we must
  show that a point $P\in E_2(\F_p)$ is in the image of $\tau$ if and
  only if $\psi_{\tau}([P],T_2)=1$. By bilinearity, $\psi_{\tau}([P],
  kT_2) = \psi_{\tau}([P], T_2)^k$. As~$E_2[\hat{\tau}]$ is generated
  by $T_2$, $P$ pairs trivially with $T_2$ if and only if it pairs
  trivially with every element of $E_2[\hat{\tau}]$. By the left
  non-degeneracy of $\psi_\tau$, this occurs if and only if $[P]$
  represents the trivial class of $E_2(\F_p)/\tau(E_1(\F_p))$, i.e.,
  $P$ is in the image of $\tau$.

  To prove the second equality, we now compute $\psi_{\tau}([P],T_2)$
  explicitly. Functions $f$ and $g$ whose divisors are
  given in \eqref{eq:divisors} are $f(x,y) = x$, $g(x,y) =
  \frac{y}{x}$. No scaling is necessary, as we have
\begin{equation*}
f( \tau(P)) = x(\tau(P)) = \frac{y^2}{x^2} = g^2(P).
\end{equation*}
Thus, $\psi_\tau([P], T_2)$ is a square in $\F_p^\times$ if and only
if $( \frac{x}{p} ) = 1$. This proves the second equality when $[P] \neq
[T_2], [\infty]$. The result is trivial for $[P] = [\infty]$. It remains to
prove $\psi_\tau( [T_2], T_2 ) = ( \frac{r}{p} )$. Notice that $T_2
\in \tau( E_1(\F_p))$ if and only if $E_1[2] \subseteq
E_1(\F_p)$. From this we see 
\begin{eqnarray*}
\psi_\tau( [T_2], T_2) = +1 & \Longleftrightarrow & E_1[2]
\subseteq E_1(\F_p) \\
& \Longleftrightarrow & x^2 + ax + b \mbox{ splits in $\F_p[x]$} \\
& \Longleftrightarrow & r = a^2 - 4b \equiv \square \pmod{p} \\
& \Longleftrightarrow & ( \tfrac{r}{p} ) = +1.
\end{eqnarray*}
\end{proof}
\begin{Remark}
In fact, $[T_2] = [\infty]$ if and only if $\bigl( \frac{r}{p} \bigr) = 1$.
\end{Remark}
\subsection*{Dual Isogeny Formulas}
Finally, we note that we can provide an equally explicit result for
the pairing attached to the dual $\hat{\tau}$. The cokernel character
$\chi_{\hat{\tau}}$ is defined by
\begin{equation*}
\chi_{\hat{\tau}}(P) = \left\{ \begin{array}{rcl}
+1 & & P \in \hat{\tau}( E_2(\F_p)) \\
-1 & & P \notin \hat{\tau}( E_2(\F_p)).
\end{array} \right.
\end{equation*}
Let $\hat{f}$ and $\hat{g}$ be the functions on $E_1$ whose
divisors are
\begin{equation*}
\begin{split}
(\hat{f}) & = 2(T_1) - 2(\infty), \\
(\hat{g}) & = \hat{\tau}^* \bigl( (T_1) - (\infty) \bigr).
\end{split}
\end{equation*}
For example, take $\hat{f}(x,y) = x$, and $\hat{g}(x,y) =
\frac{y}{2x}$. Then the associated pairing \begin{equation*}
\psi_{\hat{\tau}} \colon \frac{ E_1(\F_p) }{ \hat{\tau}(E_2(\F_p)) }
\times E_1[\tau] \longrightarrow \frac{\F_p^\times}{\F_p^{\times 2}}
\end{equation*}
may be computed via
\begin{equation*}
\psi_{\hat{\tau}} ( [S], kT_1 ) = \left\{ \begin{array}{ccl}
\hat{f}(S)^k & & [S] \notin \{[T_1], [\infty] \} \\
\left( \frac{ \hat{f}(S + Q) }{ \hat{f}(Q) } \right)^k & & [S] \in
\{[T_1], [\infty] \}.
\end{array} \right.
\end{equation*}
\begin{Prop}\label{prop:chi_psi_dual}
For any $P = (x,y) \in E_1(\F_p)$,
\begin{equation*}
\chi_{\hat{\tau}} = \psi_{\hat{\tau}}( [P], T_1) =
\left\{ \begin{array}{ccl} 
( \tfrac{x}{p} ) & & [P] \neq [T_1], [\infty] \\
( \tfrac{b}{p} ) & & [P] = [T_1] \\
1 & & [P] = [\infty].
\end{array} \right.
\end{equation*}
\end{Prop}
\begin{proof}
The argument parallels Proposition \ref{prop:chi_psi} exactly. 
\end{proof}

\section{Weighted character sums}

There are well-established connections between arithmetic data and
character sums arising from elliptic curves. If $y^2 = f(x)$ is an
integral model for an elliptic curve over $\Q$ with complex
multiplication, and $p$ is a prime of good reduction, then the work of
Deuring \cite{Deuring:1941} demonstrates that character sums of the
form
\begin{equation*}
\sum_{x=1}^{p-1} \left( \tfrac{ f(x)}{p} \right)
\end{equation*}
can be computed in terms of the trace of Frobenius and the splitting
of $p$ in the endomorphism ring of the curve. Similar results have
been established by many different authors. We mention, for example,
the works of Williams \cite{Williams:1979}, Joux-Morain
\cite{Joux:1995}, and Padma-Venkataraman \cite{Padma:1996}, which each
take slightly different approaches to such character sums, consolidate
many previous results, and contain comprehensive bibliographies.

We briefly contrast these character sums to those contained in the
present paper. First, the character $\chi_\tau$ is determined by an
isogeny, not an elliptic curve. Second, the terms of the sum are
weighted by a non-trivial integer-valued function\footnote{This is
  necessary if we hope to obtain interesting results, since trivially
  we have $\sum_P \chi_\tau(P) = 0$.}. To the authors' knowledge,
these sums have not been extensively studied.

We recall our convention to use $\{ \cdot \}$ to denote the lifting
$\F_p \to \Z \cap [0,p)$, and introduce a second convention that a
primed sum over the points on an elliptic curve will exclude the point
at infinity on that curve. Under these conventions, we seek to
evaluate the sum
\begin{equation}\label{chisums}
\psum{P \in E_2(\F_p)} \{x(P) - a\} \chi_\tau(P),
\end{equation}
where $\tau$ is the isogeny given in \eqref{eq:tau_formulas}. Here, we
are taking $E_1$, $E_2$, and $\tau$ to be defined over $\F_p$.

We first evaluate a useful character sum. For a given prime $p$ and $k
\in \Z$ such that $p \nmid k$, we define
\begin{equation*}
\delta_k := \tfrac{1}{2} \left( 1 + \bigl( \tfrac{k}{p} \bigr) \right) 
= \left\{ \begin{array}{rcl}
1 & & k \equiv \square \pmod{p}, \\
0 & & k \not\equiv \square \pmod{p}.
\end{array} \right.
\end{equation*}
\begin{Lemma}\label{lemma:quadcalc}
For any integer $k$ relatively prime to $p$,
\begin{equation}\label{eq:quad_sum}
\sum_{u=1}^{p-1} u \left( \frac{u^2 + k}{p} \right) = -p\delta_k.
\end{equation}
\end{Lemma}

\begin{proof}
Let $S$ be the sum. Then by substituting $u \mapsto p-u$, we find
\begin{equation*}
S = \sum_{u=1}^{p-1} ( p-u ) \left( \frac{(p-u)^2+k}{p} \right) = p
\sum_{u=1}^{p-1} \left( \frac{u^2+k}{p} \right) - S.
\end{equation*}
Therefore,
\begin{equation}\label{eq:2S_over_p}
\begin{split}
\frac{2S}{p} & = \sum_{u=1}^{p-1} \left( \frac{u^2+k}{p} \right) =
-\left( \frac{k}{p} \right) + \sum_{u=0}^{p-1} \left( \frac{u^2+k}{p}
\right) \\
& = -\left( \frac{k}{p} \right) - p + \sum_{u=0}^{p-1} \left[ 1
  + \left( \frac{u^2+k}{p} \right) \right].
\end{split}
\end{equation}
Let $C/\F_p$ be the conic $u^2 + kw^2 = v^2$. The final term in
\eqref{eq:2S_over_p} counts the number of $\F_p$-rational points on
$C$ within the affine region $w \neq 0$.  Since $(p, k) = 1$, the conic
is birational to $\P^1$ and has $p+1$ points. Exactly two of these,
$(1: \pm 1 : 0)$, lie on $w = 0$, so the affine region contains~$p-1$ points. 
Thus
\begin{equation*}
  S = \frac{p}{2} \left[ -\bigl( \tfrac{k}{p} \bigr) - p + (p-1) \right]
  = -\frac{p}{2} \left( 1 + \bigl( \tfrac{k}{p} \bigr) \right) = -p\delta_k,
\end{equation*}
which completes the proof.
\end{proof}

\begin{Remark}
  Sums of the form \eqref{eq:quad_sum} will appear in the proof below,
  with $u = x - a$. This motivates the choice $g = x(P) - a$ as the
  weight in \eqref{chisums}. We mention in passing that this function
  $g$ has an interesting geometric description: If $Q_i = (\alpha_i,
  0)$ are the $2$-torsion points on~$E_2$ which are not in the kernel
  of $\hat{\tau}$, then $a = \frac{1}{2}(\alpha_1 + \alpha_2)$. That
  is, we may think of $g = 0$ as the unique vertical line which
  intersects the $x$-axis at a point equidistant to both $Q_1$ and
  $Q_2$.
\end{Remark}

As always, let $p>3$ be prime.  We have already seen at the start of
\S2 that we may transform any $\F_p$-rational $2$-isogeny into the
isogeny $\tau$ of \eqref{eq:tau_formulas}. Then the isogeny $\tau$ is
between the curves
\begin{equation*}
\begin{split}
  E_1 \colon & \qquad y^2 = f_1(x) = x^3 + \phantom{2}ax^2 + bx,
  \qquad a, b \in \F_p, \\ 
  E_2 \colon & \qquad y^2 = f_2(x) = x^3 - 2ax^2 + rx, \qquad r = a^2
  - 4b.
\end{split}
\end{equation*}
Recall that $T:=(0,0)\in E_2(\F_p)$ generates the kernel of
$\hat{\tau}$.  Our goal is to evaluate the sum
\begin{equation}\label{eq:Stau_def} 
S_\tau := \psum{P \in E_2(\F_p)} \bigl\{ x(P) - a \bigr\} \chi_\tau(P).
\end{equation}
For convenience, we define an error term:
\begin{equation*}
R_{a,b} := \delta_{-b} - \sum_{x=1}^{\{a\}-1} \bigl( \tfrac{x}{p}
\bigr).
\end{equation*}
\begin{Prop}\label{bigprop}
  Let $p>3$ be prime, and $E_1$, $E_2$, $\tau$, and $S_\tau$ as above.
  Then $S_\tau$ is divisible by $p$ and
\begin{equation}\label{eq:Stau}
-\frac{1}{p} S_\tau = h_p^* + R_{a,b}.
\end{equation}
Further, $\left| R_{a,b} \right| \leq \{a\}$.
\end{Prop}
\begin{proof}
  The bound on $R_{a,b}$ is immediate. As $h_p^*$ and $R_{a,b}$ are
  integers, it remains only to establish \eqref{eq:Stau}. Applying
  Proposition \ref{prop:chi_psi}, we have:
\begin{equation*}
\begin{split}
S_{\tau} & = \psum{P \in E_2(\F_p)} \bigl( x(P) - a \bigr)
\chi_\tau(P) \\
& = \{x(T)-a \}\,\chi_\tau(T) + \sum_{\substack{P \in
E_2(\F_p) \\ P \neq T, \infty}} 
\bigl\{ x(P) - a \bigr\} \bigl( \tfrac{x(P)}{p} \bigr) \\
& \hphantom{\{p-a\} \bigl( \tfrac{r}{p} \bigr) + \sum_{x=1}^{p-1} \{x - a\} \bigl(
\tfrac{x}{p} \bigr) + \sum_{x=1}^{p-1} \{x-a\} \bigl( \tfrac{ (x-a)^2 -
  4b }{p} \bigr)} 
\end{split}
\end{equation*}
\begin{equation*}
\begin{split}
& = \{p-a\} \bigl( \tfrac{r}{p} \bigr) + \sum_{x=1}^{p-1} \{x - a\} \bigl(
\tfrac{x}{p} \bigr) \bigl( 1 + \bigl( \tfrac{x^3 - 2ax^2 
  + rx}{p} \bigr) \bigr) \\
& = \{p-a\} \bigl( \tfrac{r}{p} \bigr) + \sum_{x=1}^{p-1} \{x - a\} \bigl(
\tfrac{x}{p} \bigr) + \sum_{x=1}^{p-1} \{x-a\} \bigl( \tfrac{ (x-a)^2 -
  4b }{p} \bigr) \\
& = \sum_{x=1}^{p-1} \{x-a\} \bigl( \tfrac{x}{p} \bigr) +
\sum_{x=0}^{p-1} \{x-a\} \bigl( \tfrac{ (x-a)^2 - 4b }{p} \bigr).
\end{split}
\end{equation*}
We evaluate these two sums in turn. First,
\begin{equation*}
\begin{split}
\sum_{x=1}^{p-1} \{x - a\} \bigl( \tfrac{x}{p} \bigr) &
= \sum_{x=1}^{\{a\}-1} (p + \{x\}- \{a\}) \bigl(
\tfrac{x}{p} \bigr) + \sum_{x=\{a\}}^{p-1} (\{x\}-\{a\})
\bigl( \tfrac{x}{p} \bigr) \\
& = p \sum_{x=1}^{\{a\}-1} \bigl( \tfrac{x}{p} \bigr) + \sum_{x=1}^{p-1}
x \bigl( \tfrac{x}{p} \bigr) - a\sum_{x=1}^{p-1} \bigl( \tfrac{x}{p} \bigr) \\
& = p \sum_{x=1}^{\{a\}-1} \bigl( \tfrac{x}{p} \bigr) - ph_p^*.
\end{split}
\end{equation*}
We note that the second sum is unchanged by making the substitution $u=x-a$. Thus,
\begin{equation*}
\sum_{x=0}^{p-1}\{x-a\}\bigl(\tfrac{(x-a)^2-4b}{p}\bigr) = \sum_{u=0}^{p-1}\{u\}\bigl(\tfrac{u^2-4b}{p}\bigr)=-p\delta_{-b},
\end{equation*}
by Lemma \ref{lemma:quadcalc}.  Combining the two sums, we have
\begin{equation*}
S_\tau = -ph_p^* - p \delta_{-b} + p \sum_{x=1}^{\{a\}-1} \bigl(
  \tfrac{x}{p} \bigr),
\end{equation*}
which gives \eqref{eq:Stau}.
\end{proof}
The Main Theorem can be viewed as the global version of Proposition
\ref{bigprop}. Consider elliptic curves $E_1/\Q$ and $E_2/\Q$ with
respective Weierstrass models
\begin{equation*}
\begin{split}
E_1 \colon y^2 & = x^3 + \phantom{2}ax^2 + bx, \\
E_2 \colon y^2 & = x^3 - 2ax^2 + (a^2-4b)x,
\end{split}
\end{equation*}
with $a, b \in \Z$. As always, we let $\tau \colon E_1 \rightarrow
E_2$ be the explicit $\Q$-rational 2-isogeny given in
\eqref{eq:tau_formulas}.

\begin{Theo}\label{thm:main}
Let $\tau \colon E_1 \to E_2$ be as above, and let $p>3$ be a prime of
good reduction.
\begin{enumerate}[(a)]
\item For all such $p$, the weighted character sum 
\begin{equation*}
S_{\tau, p} := \psum{P \in E_2(\F_p)} \bigl\{ x(P)-a \bigr\} \chi_\tau(P)
\end{equation*}
is divisible by $p$.
\item $S_{\tau, p}$ approximates $-p h_p^*$ in the following sense: the
  quantity
\begin{equation*}
R_{a,b}(p) = -\frac{1}{p} S_{\tau, p} - h_p^*
\end{equation*}
is bounded in absolute value by a constant $C_\tau$, independent of
$p$.
\item If there exists one $p>|a|$ such that $R_{a,b}(p) = 0$, then there
  exists a set of primes of positive density for which $S_{\tau, p} =
  h_p^*$, determined by explicit congruence conditions.
\end{enumerate}
\end{Theo}
\begin{proof}
  Part (a) is precisely Proposition \ref{bigprop} applied to any prime
  of good reduction. For part (b), it is enough to bound the number of
  terms in $R_{a,b}$ independently of $p$. If $a > 0$, this is
  obvious; the sum has exactly $\{a \}$ terms, and $\{a \} \leq a$.
When $a = 0$ we trivially have $C_\tau = 1$. Now suppose $a < 0$. It
is no obstruction to assume $p > |a|$, and so $\{a \} = p + a$. We have
\begin{equation*}
  R_{a,b} = \delta_{-b} - \sum_{x=1}^{ \{a \} - 1 } \bigl( \tfrac{x}{p} \bigr) 
  = \delta_{-b} - \sum_{x=1}^p \bigl( \tfrac{x}{p} \bigr) + \sum_{x=\{a\}}^p \bigl( \tfrac{x}{p} \bigr) = \delta_{-b} + \sum_{x = p+a}^p \bigl( \tfrac{x}{p} \bigr),
\end{equation*}
and so $|R_{a,b}| \leq |a| + 2$.

For part (c), suppose $p_0 > |a|$ is a prime for which $R_{a,b}(p_0) =
0$. So $p = p_0$ is a solution to:
\begin{equation}\label{eq:partc_sum}
\delta_{-b} = \sum_{x=1}^{ \{a\}-1} \bigl( \tfrac{x}{p} \bigr).
\end{equation}
Now, among $p > |a|$, the individual terms of the sum
\eqref{eq:partc_sum} never vanish, and so the sum is periodic as a
function of $p$, with respect to some sufficiently large modulus,
e.g., $N = 4G$, where $G$ is the least common multiple of $\{2, 3,
\dots, |a| - 1 \}$. Thus, every prime in the sequence $\{p_0 + kN\}$
also satisfies $R_{a,b}(p) = 0$.
\end{proof}
\noindent The analogous result for the dual isogeny $\hat{\tau}$ is
proved in the appendix.
\begin{Remark}
  It is natural to ask what happens when the hypothesis in part (c)
  does not hold. It is possible that the error term is zero for only
  finitely many $p$. For example, when $(a,b) = (9,-1)$, the error
  term vanishes only for $p = 7$. It is even possible that the error
  term is never zero (e.g., Example 2 below). Indeed, note that by the
  periodicity modulo $N$, the error term is never zero if it is
  non-zero for all $p < N$. 
\end{Remark}

\begin{Remark}
  For fixed $a$ and varying $p$, it is not hard to argue that the
  uniform bounds constructed in the proof of Theorem \ref{thm:main} are
  the best possible. However, for a \emph{particular} value of $p$,
  better bounds certainly exist. For example, if $a = O(\log p)$, then
  the estimates of P\'{o}lya-Vinogradov (and later improvements by
  Burgess) may offer substantial improvement. For details, see for
  example \cite[\S9.4]{Montgomery:2007}.
\end{Remark}

\subsection*{Examples}


We consider a few examples for illustration. In all cases, the
selection of values $(a,b)$ determines elliptic curves $E_1$, $E_2$ by
\eqref{eq:E1E2} and an isogeny $\tau$ by \eqref{eq:tau_formulas}.

\begin{Exam}\label{ex:intro2}
Set $(a,b) = (2,-1)$. Then we have $R_{a,b}(p) = 0$ for all
primes $p>3$.  Hence, $-\frac{S_\tau}{p}=h_p^*$ for all such $p$. 
\end{Exam}
\begin{Exam}
Set $(a,b) = (3,-1)$. We find $R_{a,b} = -\bigl(
\tfrac{2}{p} \bigr)$, which is non-zero for all $p > 3$.
\end{Exam}


\begin{Exam}
Set $(a,b) = (7,2)$. Then, for $p\neq 5$ ($R_{a,b}(5) \neq 0$ by a separate calculation), 
\begin{equation*}
R_{a,b} = \delta_{-2} - \sum_{x=1}^6 \bigl( \tfrac{x}{p} \bigr) = 2 +
\delta_{-2} + \bigl( \tfrac{2}{p} \bigr) + \bigl( \tfrac{3}{p} \bigr)
+ \bigl( \tfrac{5}{p} \bigr) + \bigl( \tfrac{6}{p} \bigr).
\end{equation*}
We wish to decide when this sum vanishes.  By a parity argument, we
must have $\delta_{-2}=0$, and of the four remaining Legendre symbols,
three must evaluate to $-1$, and one to $+1$.  Further, we cannot have
$\bigl( \tfrac{5}{p} \bigr)=1$. Otherwise, $R_{a,b} = 0$ implies
$\bigl( \tfrac{2}{p} \bigr) = \bigl( \tfrac{3}{p} \bigr) = \bigl(
\tfrac{6}{p} \bigr)=-1$, contradicting the multiplicativity of the
Legendre symbol. The three choices for which of the other Legendre
symbols is to be $-1$ each lead to congruence conditions easily
determined via quadratic reciprocity.  Namely, we find:
\begin{equation*}
-\frac{S_\tau}{p}=h_p^* \quad\iff\quad p\equiv 17,43,67,83,107,\text{ or }113\mod 120.
\end{equation*}
\end{Exam}

\section{Complex Multiplication}\label{section:CM}

Having computed the character sums associated to $2$-isogenies given
in the specific form \eqref{eq:tau_formulas}, we turn to the analogous
computation for other $2$-isogenies. The principal technical lemma is
that the character sums are unaffected by applying a change of
coordinates $x \mapsto x - \varepsilon$ to the curve.  This will allow
us to change the codomain of an arbitrary 2-isogeny to one of the form
in \eqref{eq:tau_formulas}, and allow us to compute the corresponding
character sums.  As a primary application, the elliptic curves with
complex multiplication by $\sqrt{-1}$, $\sqrt{-2}$, or $\sqrt{-7}$
possess an endomorphism of degree $2$, and in this section we will
compute the associated weighted character sum. In particular, we
deduce Proposition \ref{-2ex}.

Suppose that $\phi \colon E_1' \to E_2'$ is a degree $2$-isogeny
defined over $\F_p$ with $p > 3$. Then there exist $\F_p$-isomorphisms
$\alpha_i \colon E_i' \to E_i$ and an isogeny $\tau$ of the form
\eqref{eq:tau_formulas} such that the left square of the following
diagram commutes:
\begin{equation}\label{eq:phi_tau_diagram}
\xymatrix{
E_1' \ar[d]_{\alpha_1} \ar[r]^{\phi} & E'_2 \ar[d]^{\alpha_2} \ar[r]^{\chi_\phi}& \bmu_2 \\
E_1 \ar[r]_{\tau} & E_2 \ar[r]_{\chi_\tau} & \bmu_2 \ar@{=}[u]_{\mathrm{id}}}
\end{equation}
Here, $\chi_\phi$ and $\chi_\tau$ are the characters attached $\phi$
and $\tau$, respectively.
\begin{Lemma}\label{lemma:phi_tau_sums}
  The right square of the above diagram commutes, i.e., $\chi_\phi(P')
  = \chi_\tau \bigl( \alpha_2(P') \bigr)$ for any $P' \in E_2'(\F_p)$.
\end{Lemma}
\begin{proof}
This is a simple diagram chase. If $\chi_\phi(P') = 1$, then there
exists a point $Q' \in E_1'(\F_p)$ such that $\phi(Q') = P'$. Hence,
$\alpha_2(P') = \tau(\alpha_1(Q))$, and so $\chi_\tau(\alpha_2(P')) =
1$ also. This argument is easily reversed, and the result follows.
\end{proof}
As a consequence, we demonstrate the weighted character sum
\begin{equation*}
S_\phi := \psum{P \in E_2'(\F_p)} \bigl\{ x(P) - \xi \bigr\} \chi_\phi(P)
\end{equation*}
equals the sum $S_\tau$ of \eqref{eq:Stau_def}, provided $\alpha_2$ is
of a particular form and $\xi$ is chosen appropriately. Let $E_2'$ be
the elliptic curve
\begin{equation}\label{eq:E2prime}
y^2 = (x - \varepsilon)(x^2 - 2\mu x + \nu), \qquad \varepsilon, \mu, \nu \in \F_p.
\end{equation}
defined over $\F_p$. Let $\phi$ be an isogeny such that the kernel of
the dual isogeny $\hat{\phi}$ is generated by $(\varepsilon, 0) \in
E_2'(\F_p)$. Set $(a,b) = \bigl( \mu - \varepsilon, \frac{1}{4}(\mu^2
- \nu)\bigr)$, and let $E_1, E_2, \tau$ be as in
\eqref{eq:tau_formulas}. Then the map
\begin{equation}\label{eq:alpha2}
\alpha_2(x,y) := (x - \varepsilon, y)
\end{equation}
has the property that $\hat{\tau} \circ \alpha_2$ and $\hat{\phi}$
have the same kernel. This guarantees the existence of a unique
$\alpha_1$ satisfying $\hat{\tau} \circ \alpha_2 = \alpha_1 \circ
\hat{\phi}$, and in fact, $\alpha_1$ completes the diagram
\eqref{eq:phi_tau_diagram}. Now set $\xi = \varepsilon + a$, so that
\begin{equation*}
S_\phi = \psum{P \in E_2'(\F_p)} \bigl\{ x(P) - \varepsilon - a \bigr\}
\chi_\phi(P).
\end{equation*}
\begin{Lemma}
$S_\phi = S_\tau$.
\end{Lemma}
\begin{proof}
Let $p$ be a prime of good reduction. Applying the previous lemma, we have
\begin{equation}\label{eq:translate}
\begin{split}
S_\phi & = \psum{P \in E_2'(\F_p)} \{x(P) - \varepsilon - a \}\, \chi_\phi(P)  \\
& = \psum{P \in E_2(\F_p)} \left\{ x \bigl( \alpha_2^{-1}(P) \bigr) -
  \varepsilon - a \right\} \chi_\phi \bigl( \alpha_2^{-1}(P)
\bigr) \\
& = \psum{P \in E_2(\F_p)} \{x(P) - a \}\, \chi_\tau(P) = S_\tau.
\end{split}
\end{equation}
\end{proof}
We now address CM endomorphisms. For each endomorphism $\phi$, we
select isomorphisms $\alpha_i$ such that $\alpha_2 \circ \phi = \tau
\circ \alpha_1$, as in diagram \eqref{eq:phi_tau_diagram}, and such
that $E_2'$ and $\alpha_2$ have the forms \eqref{eq:E2prime} and
\eqref{eq:alpha2}, respectively. Evaluation of $S_\phi$ now follows
from the previous lemma.
\subsection{CM by $-1$}
Consider first the elliptic curve $E/\Q \colon y^2 = x^3 + x$, which
has complex multiplication by $\Z[\sqrt{-1}]$, and possesses the
degree 2 endomorphism
\begin{equation}\label{eq:CMneg1_endo}
\phi \colon E \longrightarrow E, \qquad \phi(x,y) = \left( \frac{x^2 +
    1}{2ix}, \frac{ y(x^2-1) }{ (2i-2)x^2 } \right).
\end{equation}
Let $E'/\Q$ be the elliptic curve with Weierstrass equation $y^2 = x^3 -
\frac{1}{4}x$, and let $\tau \colon E' \to E$ be the isogeny
corresponding to $(a,b) = (0, -\frac{1}{4})$. If $\alpha_2$ is the
identity and $\alpha_1$ is the isomorphism
\begin{equation*}
\alpha_1(x,y) = (u^{-2}x, u^{-3}y), \qquad u = -(1+i),
\end{equation*}
then $\alpha_2 \circ \phi = \tau \circ \alpha_1$. Suppose $p > 3$ is a
prime of good reduction which splits in $\Q(i)$. Each of the morphisms  
in the above diagram has a well-defined reduction mod $p$, and the
diagram still commutes after reduction. Applying Lemma
\ref{lemma:phi_tau_sums} and Theorem \ref{thm:main}, we see
\begin{equation*}
S_\phi = S_\tau = -ph_p^* -pR_{0, -\frac{1}{4}} = -p.
\end{equation*}
This proves:
\begin{Cor}\label{cor:CMneg1}
Let $p > 3$ be a prime, and suppose $p \equiv 1 \pmod{4}$. Then
\begin{equation*}
-\frac{1}{p} \psum{P \in E(\F_p)} \{ x(P) \} \chi_\phi(P) = 1.
\end{equation*}
\end{Cor}
\subsection{CM by $-2$}
Next, we compute the character sum associated to a degree $2$
endomorphism on an elliptic curve with complex multiplication by
$\Z[\sqrt{-2}]$. One such elliptic curve is  
\begin{equation*}
E/\Q \colon y^2 = (x+2)(x^2-2).
\end{equation*}
One degree $2$ endomorphism on $E$ is:
\begin{equation}\label{eq:tau_CMneg2}
\phi(x,y) = \left( \frac{ (x+2)^2 + 2}{ -2(x+2) }, \frac{ y \bigl( (x+2)^2
    -2 \bigr) }{2\sqrt{-2} (x+2)^2} \right).
\end{equation}
(This may be derived from the formula given in
\cite[pg. 111]{Silverman:AEC2}, but note that we are using a different
Weierstrass equation.) Let $p > 3$ be a prime that splits in
$\Q(\sqrt{-2})$, so that $\bigl( \tfrac{-2}{p}
\bigr)=1$. Equivalently, $p$ is congruent to $1$ or $3 \pmod{8}$. For
such primes, there is a well-defined reduction for $\phi$ over $\F_p$,
which we will also denote $\phi$. We now prove Proposition \ref{-2ex}. Set
\begin{equation*}
S_\phi := \psum{P \in E(\F_p)} \{ x(P) \} \chi_\phi(P).
\end{equation*}
\begin{Cor}
  For any prime $p > 3$ such that $p \equiv 1$ or $3 \pmod{8}$,
  $-\frac{1}{p} S_\phi = h_p^*$.
\end{Cor}
\begin{proof}
Let $E_1/\Q$ and $E_2/\Q$ denote the elliptic curves corresponding to
the choice $(a,b) = (2,\frac{1}{2})$ in \eqref{eq:tau_formulas}; let
$\tau$ denote the corresponding $2$-isogeny. It is straightforward to
find isomorphisms $\alpha_i$ with the requisite properties such that
$\tau \circ \alpha_1 = \alpha_2 \circ \phi$. Hence, $S_\phi =
S_\tau$. By Proposition \ref{bigprop}, $-\frac{1}{p}S_\tau = h_p^* +
R_{2,\frac{1}{2}} = h_p^*$, as $R_{2,\frac{1}{2}} = 0$ if $p \equiv 1, 3 \pmod{8}$.
\end{proof}

\subsection{CM by $-7$}
Finally, we compute the character sum associated to a degree $2$
endomorphism on a curve with complex multiplication by
$\Z[\sqrt{-7}]$. We take the elliptic curve
\begin{equation*}
E/\Q \colon y^2 = (x+7)(x^2-7x+14).
\end{equation*}
Let $p > 3$ be a prime which splits in $\Q(\sqrt{-7})$, so $\bigl(
\tfrac{-7}{p} \bigr)=1$. Set $\beta = \frac{1+\sqrt{-7}}{2}$.  A
degree $2$ endomorphism on $E$ is given explicitly by
(\cite[pg. 111]{Silverman:AEC2}):
\begin{equation}\label{eq:tau_CMneg7}
  \phi(x,y) = \left( \beta^{-2} \left( x - \frac{7(1 - \beta^4)}{x + \beta^2 - 2} \right),\beta^{-3} y \left(1 + \frac{7(1 - \beta)^4}{(x + \beta^2 - 2)^2} \right) \right).
\end{equation}
Let $E_1$ and $E_2$ be the elliptic curves
\begin{equation*}
\begin{split}
E_1/\Q(\sqrt{-7}) \colon y^2 & = x^3 + \phantom{2}ax^2 + bx, \\
E_2/\Q(\sqrt{-7}) \colon y^2 & = x^3 - 2ax^2 + rx,
\end{split}
\end{equation*}
where
\begin{equation*}
a = \tfrac{3}{2} \bigl(\beta - 4 \bigr), \qquad b = \tfrac{7}{16} \bigl( 3\beta + 14 \bigr), \qquad r = a^2 - 4b.
\end{equation*}
The curves $E$, $E_1$, and $E_2$ are all isomorphic over
$\Q(\sqrt{-7})$. Let $\alpha_1$ and $\alpha_2$ be the isomorphisms
\begin{equation*}
\begin{split}
  \alpha_1 \colon E \rightarrow E_1, \qquad (x,y) & \mapsto \bigl( (\beta - 1)^2 x + \beta + 3, (\beta - 1)^3 y \bigr), \\
  \alpha_2 \colon E \rightarrow E_2, \qquad (x,y) & \mapsto (x + 4 -
  \beta, y).
\end{split}
\end{equation*}
As before, we have $\tau \circ \alpha_1 = \alpha_2 \circ \phi$, and
each morphism reduces to an $\F_p$-rational morphism on the reduced
curves, as $p$ splits in $\Q(\sqrt{-7})$. If we define
\begin{equation*}
S_\phi := \psum{P \in E(\F_p)} \{ x(P) \} \chi_\phi(P),
\end{equation*}
then $S_\phi = S_\tau$, and we find $-\frac{1}{p} S_\phi = h_p^* +
R_{a,b}$.

Unlike all previous examples, however, we do not have a uniform bound
for the error. Our bound is always in terms of $\{a \}$, but in this
situation $a \not\in \Z$. Hence, as a function of $p$, $\{a \}$ is not
eventually constant, and the previous results bounding $R_{a,b}$ do
not apply.

\section*{Appendix: A proof of a dual isogeny calculation}
Let $\tau$ and $\hat{\tau}$ be the isogenies defined in
\eqref{eq:tau_formulas}, with $a,b\in \F_p$ for $p>3$. Recall that we
set $r = a^2 - 4b$. Let $\eta=\{\frac{a}{2}\}$, and define
\begin{equation*}
\hat{R}_{a,b} := \delta_{-r} + \sum_{x=1}^{\eta} \bigl( \tfrac{-x}{p}
\bigr).
\end{equation*}
We wish to evaluate
\begin{equation*}
S_{\hat{\tau}} := \psum{P \in E_1(\F_p)} \bigl\{ x(P) + \tfrac{a}{2}
\bigr\}\, \chi_{\hat{\tau}}(P).
\end{equation*}
\begin{Theo}
For any odd prime $p$ of good reduction,
\begin{equation}\label{eq:Stauhat}
-\frac{1}{p} S_{\hat{\tau}} = h_p^* + \hat{R}_{a,b}.
\end{equation}
\end{Theo}
\begin{proof}
By Proposition \ref{prop:chi_psi_dual}, we first obtain:
\begin{equation*}
\begin{split}
S_{\hat{\tau}} & = \eta \bigl( \tfrac{b}{p} \bigr) +
\sum_{\substack{P\in E_1(\F_p) \\ P \neq T_1, \infty}} \bigl\{ x(P) + \tfrac{a}{2}
\bigr\} \chi_{\hat{\tau}}(P) \\
& = \eta \bigl( \tfrac{b}{p} \bigr) + \sum_{x=1}^{p-1} \bigl\{ x +
\tfrac{a}{2} \bigr\} \bigl( \tfrac{x}{p} \bigr) \bigl(1 + \bigl(
\tfrac{ x^3 + ax^2 + bx}{p} \bigr) \bigr) \\
& = \eta \bigl( \tfrac{b}{p} \bigr) + \sum_{x=1}^{p-1} \bigl\{x +
\tfrac{a}{2} \bigr\} \bigl( \tfrac{x}{p} \bigr) + \sum_{x=1}^{p-1}
\bigl\{ x + \tfrac{a}{2} \bigr\} \bigl( \tfrac{x^2 + ax + b}{p}
\bigr) \\
& = \sum_{x=1}^{p-1} \bigl\{ x + \tfrac{a}{2} \bigr\} \bigl(
\tfrac{x}{p} \bigr) + \sum_{x=0}^{p-1} \bigl\{ x + \tfrac{a}{2}
\bigr\} \bigl( \tfrac{x^2 + ax + b}{p} \bigr).
\end{split}
\end{equation*}
The first sum evaluates as follows:
\begin{equation*}
\begin{split}
\sum_{x=1}^{p-1} \bigl\{ x + \tfrac{a}{2} \bigr\} \bigl(
\tfrac{x}{p} \bigr) & = \sum_{x=1}^{p-1-\eta} (\{x\} +
\{\eta\}) \bigl( \tfrac{x}{p} \bigr) + \sum_{x=p-\eta}^{p-1}
(\{x\} + \{\eta\} - p) \bigl( \tfrac{x}{p} \bigr) \\
& = \sum_{x=1}^{p-1}  x\bigl( \tfrac{x}{p} \bigr) + \eta
\sum_{x=1}^{p-1} \bigl( \tfrac{x}{p} \bigr) - p \sum_{x=p-\eta}^{p-1}
\bigl( \tfrac{x}{p} \bigr) = -ph_p^* - p \sum_{x=1}^\eta \bigl( \tfrac{-x}{p} \bigr).
\end{split}
\end{equation*}
As for the second sum, we substitute $u=x+\eta$ and apply Lemma \ref{lemma:quadcalc} to get:
\begin{equation*}
\sum_{x=0}^{p-1} \bigl\{x + \tfrac{a}{2} \bigr\} \bigl( \tfrac{x^2
  + ax + b}{p} \bigr) = \sum_{u=0}^{p - 1} \{u\} \left( \frac{u^2 -
    \tfrac{r}{4}}{p} \right)  = -p\delta_{-r/4} = -p\delta_{-r}.
\end{equation*}
Combining the above calculations yields the result.
\end{proof}

We conclude with a brief comment on the ``globalization'' of the error
term for $S_{\hat{\tau}}$, analogously to Theorem~\ref{thm:main} for
$S_\tau$.  Take $E_1$, $E_2$, and $\tau$ as in
\eqref{eq:tau_formulas}, with $a,b\in \Z$.  For simplicity, suppose $p
> a> 0$, so that $a = \{a\}$. If $a$ is even, we have the explicit
bound $|\hat{R}_{a,b}| \leq \frac{a}{2} + 1$ for the error term. If
$a$ is odd, the behavior of the error is different. For $p > 2a$, we have
$\eta = \frac{p-1}{2} + \frac{a+1}{2}$ exactly, and so the error term
has roughly $\frac{p}{2}$ terms. Further, from \cite[Ch.~5, \S
4]{BorShaf:1966}, we have
\begin{equation*}  
\sum_{x=1}^{\frac{p-1}{2}} \left( \frac{-x}{p} \right) =
\left\{ \begin{array}{rcl}
0 & & p \equiv 1, 5 \pmod{8} \\
-3h_p & & p \equiv 3 \phantom{, 5} \pmod{8} \\
-h_p & & p \equiv 7 \phantom{, 5} \pmod{8}.
\end{array} \right.
\end{equation*}
Thus, the term $\hat{R}_{a,b}$ may dwarf $h_p$! Rather, if we define
\begin{equation*}
\hat{\rho}_{a,b} := \delta_{-r} + \sum_{x=1}^{\frac{a+1}{2}} \left(
  \frac{ -\frac{p-1}{2} - x}{p} \right),
\end{equation*}
then $|\hat{\rho}_{a,b}| \leq \frac{a+3}{2}$, and
\eqref{eq:Stauhat} can be rewritten:
\begin{equation*}
-\frac{1}{p} S_{\hat{\tau}} = \hat{\rho}_{a,b} +
\left\{ \begin{array}{rcl}
0 & & p \not\equiv 3 \pmod{8} \\
-2h_p & & p \equiv 3 \pmod{8}.
\end{array} \right.
\end{equation*}
This gives a better description of the behavior of the sum
$S_{\hat{\tau}}$ when $a$ is odd.

\section*{Acknowledgements}
We are grateful to Kirti Joshi for first bringing our
attention to the interesting divisibility properties of weighted
character sums on elliptic curves, and for many fruitful conversations
during our research. We give our thanks to Joe Silverman and Matt Papanikolas 
for helpful suggestions. We wish to recognize the contribution
of {\tt Sage} \cite{sage:2010}, which was invaluable for computational
experimentation related to this project. Finally, we appreciate the
many helpful comments and suggestions of the referee during the revision of this article.

\bibliographystyle{plain}
\bibliography{isogs}
\end{document}